\documentclass{amsart}
\usepackage{amssymb}
\usepackage{amsfonts}

\newtheorem{theorem}{Theorem}
\theoremstyle{plain}

\newtheorem{corollary}{Corollary}

\newtheorem{definition}{Definition}

\newtheorem{lemma}{Lemma}

\numberwithin{equation}{section}
\input{tcilatex}

\begin{document}
\title[Some integral inequalities]{Some integral inequalities for $\alpha $-$%
,$ $m$-$,$ $\left( \alpha ,m\right) $-logarithmically convex functions}
\author{MEVL\"{U}T TUN\c{C}$^{\square }$}
\address{$^{\square }$Department of Mathematics, Faculty of Science and
Arts, Kilis 7 Aral\i k University, Kilis, 79000, Turkey.}
\email{$^{\square }$mevluttunc@kilis.edu.tr}
\author{EBRU Y\"{U}KSEL$^{\triangledown }$}
\address{$^{\triangledown }$The Institute for Graduate Studies in Sciences
and Engineering, Kilis 7 Aral\i k University, Kilis, 79000, Turkey.}
\email{$^{\triangledown }$yuksel.ebru90@hotmail.com}
\thanks{$^{\square }$Corresponding Author}
\date{January 10, 2013}
\thanks{This paper is in final form and no version of it will be submitted
for publication elsewhere.}

\begin{abstract}
In this paper, we establish some new Hadamard type inequalities using
elementary well known inequalities for functions whose inequalities absolute
values are $\alpha $-, $m$-$,$ $\left( \alpha ,m\right) $-logarithmically
convex.
\end{abstract}

\maketitle

\section{\protect INTRODUCTION}

Let \ $f:I\subseteq
\mathbb{R}
\rightarrow
\mathbb{R}
$ be a convex mapping defined on the interval $I$\ of real numbers and $%
a,b\in I,$ with $a<b.$ The following double inequalities:%
\begin{equation*}
f\left( \frac{a+b}{2}\right) \leq \int_{a}^{b}f\left( x\right) dx\leq \frac{%
f\left( a\right) +f\left( b\right) }{2}
\end{equation*}%
hold. This double inequality is known in the literature as the
Hermite-Hadamard inequality for convex functions (see \cite{bai}-\cite{tnc}).

In this section, we will present definitions and some results used in this
paper.

\begin{definition}
Let $I$ be an interval in $%
\mathbb{R}
.$ Then $f:I\rightarrow
\mathbb{R}
,$ $\emptyset \neq I\subseteq
\mathbb{R}
$\ is said to be convex if
\begin{equation}
f\left( tx+\left( 1-t\right) y\right) \leq tf\left( x\right) +\left(
1-t\right) f\left( y\right) .  \label{7}
\end{equation}%
for all $x,y\in I$ and $t\in \left[ 0,1\right] .$
\end{definition}

In \cite{bai}, the concepts of $\alpha $-, $m$- and $\left( \alpha ,m\right)
$-logarithmically convex functions were introduced as follows.

\begin{definition}
\cite{bai}\label{d} A function $f:[0,b]\rightarrow (0,\infty )$ is said to
be $m$-logarithmically convex if the inequality%
\begin{equation}
f\left( tx+m\left( 1-t\right) y\right) \leq \left[ f\left( x\right) \right]
^{t}\left[ f\left( y\right) \right] ^{m\left( 1-t\right) }  \label{d1}
\end{equation}%
holds for all $x,y\in \lbrack 0,b]$, $m\in (0,1]$, and $t\in \lbrack 0,1]$.
\end{definition}

Obviously, if putting $m=1$ in Definition \ref{d}, then $f$ is just the
ordinary logarithmically convex on $\left[ 0,b\right] $.

\begin{definition}
\cite{tnc}\label{ddd} A function $f:[0,b]\rightarrow (0,\infty )$ is said to
be $\alpha $-logarithmically convex if%
\begin{equation}
f\left( tx+\left( 1-t\right) y\right) \leq \left[ f\left( x\right) \right]
^{t^{\alpha }}\left[ f\left( y\right) \right] ^{\left( 1-t^{\alpha }\right) }
\end{equation}%
holds for all $x,y\in \lbrack 0,b]$, $\alpha \in \left( 0,1\right] $ and $%
t\in \lbrack 0,1]$.
\end{definition}

Clearly, when taking $\alpha =1$ in Definition \ref{ddd}, then $f$ becomes
the ordinary logarithmically convex on $\left[ 0,b\right] $.

\begin{definition}
\cite{bai}\label{dd} A function $f:[0,b]\rightarrow (0,\infty )$ is said to
be $\left( \alpha ,m\right) $-logarithmically convex if%
\begin{equation}
f\left( tx+m\left( 1-t\right) y\right) \leq \left[ f\left( x\right) \right]
^{t^{\alpha }}\left[ f\left( y\right) \right] ^{m\left( 1-t^{\alpha }\right)
}  \label{d2}
\end{equation}%
holds for all $x,y\in \lbrack 0,b]$, $\left( \alpha ,m\right) \in \left( 0,1%
\right] \times \left( 0,1\right] ,$ and $t\in \lbrack 0,1]$.
\end{definition}

Clearly, when taking $\alpha =1$ in Definition \ref{dd}, then $f$ becomes
the standard $m$-logarithmically convex function on $\left[ 0,b\right] $,
and, when taking $m=1$ in Definition \ref{dd}, then $f$ becomes the $\alpha $%
-logarithmically convex function on $\left[ 0,b\right] .$

\section{NEW HADAMARD-TYPE INEQUALITIES}

\textbf{\ \ \ \ \ \ \ \ \ \ \ \ \ \ \ \ \ \ \ \ \ \ \ \ \ \ \ \ \ \ \ \ \ \
\ \ \ \ \ \ \ \ \ \ \ \ \ \ \ \ \ \ \ \ }

In order to prove our main theorems, we need the following lemma \cite{lmmm}.

\begin{lemma}
\label{l1}\cite{lmmm} Let $f:\ I\subset
\mathbb{R}
\rightarrow
\mathbb{R}
$ be a differentiable mapping on $I^{\circ }$\textit{, }$a,b\in $ $I^{\circ
} $ with $a$ $<$ $b$. If $f^{\prime }\in $ $L\left[ a,b\right] ,$ then the
following equality holds:
\begin{eqnarray}
&&\frac{f\left( a\right) +f\left( b\right) }{2}-\frac{1}{b-a}%
\int_{a}^{b}f\left( x\right) dx  \label{a} \\
&=&\frac{b-a}{2}\int_{0}^{1}\int_{0}^{1}\left[ f^{\prime }\left( ta+\left(
1-t\right) b\right) -f^{\prime }\left( sa+\left( 1-s\right) b\right) \right]
\left( s-t\right) dtds.  \notag
\end{eqnarray}
\end{lemma}

A simple proof of this equality can be also done integrating by parts in the
right hand side (see \cite{lmmm}).

The next theorems gives a new result of the upper Hermite-Hadamard
inequality for $\alpha $-, $m$-, $\left( \alpha ,m\right) $-logarithmically
convex functions.

\begin{theorem}
\label{t1}Let$\ I\supset \left[ 0,\infty \right) $ be an open interval and
let $f:\ I\rightarrow \left( 0,\infty \right) $ be a differentiable function
on $I$ such that $f^{\prime }\in L\left( a,b\right) $ for $0\leq a<b<\infty
. $ If $\left\vert f^{\prime }\left( x\right) \right\vert $ is $\left(
\alpha ,m\right) $-logarithmically convex on $\left[ 0,\frac{b}{m}\right] $
for $\left( \alpha ,m\right) \in \left( 0,1\right] ^{2},$ then
\begin{eqnarray}
&&\left\vert \frac{f\left( a\right) +f\left( b\right) }{2}-\frac{1}{b-a}%
\int_{a}^{b}f\left( x\right) dx\right\vert \\
&\leq &\left\{
\begin{array}{cc}
\frac{\left( b-a\right) }{3}\left\vert f^{\prime }\left( \frac{b}{m}\right)
\right\vert ^{m},\text{ \ \ \ \ \ \ \ \ \ \ \ \ \ \ \ \ \ \ \ \ \ \ \ \ \ \ }
& \eta =1 \\
\frac{\left( b-a\right) }{2}\left\vert f^{\prime }\left( \frac{b}{m}\right)
\right\vert ^{m}\frac{-\alpha ^{2}\ln ^{2}\eta -2\alpha \ln \eta +2\eta
^{\alpha }-2}{\alpha ^{3}\ln ^{3}\eta }, & \eta <1%
\end{array}%
\right.  \notag
\end{eqnarray}%
where $\eta =\left\vert f^{\prime }\left( a\right) \right\vert /\left\vert
f^{\prime }\left( \frac{b}{m}\right) \right\vert ^{m}.$
\end{theorem}

\begin{proof}
By Lemma \ref{l1} and since $\left\vert f^{\prime }\right\vert $ is an $%
\left( \alpha ,m\right) $-logarithmically convex on $\left[ 0,\frac{b}{m}%
\right] $, then we have%
\begin{eqnarray*}
&&\left\vert \frac{f\left( a\right) +f\left( b\right) }{2}-\frac{1}{b-a}%
\int_{a}^{b}f\left( x\right) dx\right\vert \\
&\leq &\frac{b-a}{2}\int_{0}^{1}\int_{0}^{1}\left\vert \left( f^{\prime
}\left( ta+\left( 1-t\right) b\right) \right) -\left( f^{\prime }\left(
sa+\left( 1-s\right) b\right) \right) \right\vert \left\vert s-t\right\vert
dtds \\
&\leq &\frac{b-a}{2}\left[ \int_{0}^{1}\int_{0}^{1}\left\vert s-t\right\vert
\left\vert f^{\prime }\left( a\right) \right\vert ^{t^{\alpha }}\left\vert
f^{\prime }\left( \frac{b}{m}\right) \right\vert ^{m\left( 1-t^{\alpha
}\right) }dtds\right] \\
&&+\frac{b-a}{2}\left[ \int_{0}^{1}\int_{0}^{1}\left\vert s-t\right\vert
\left\vert f^{\prime }\left( a\right) \right\vert ^{s^{\alpha }}\left\vert
f^{\prime }\left( \frac{b}{m}\right) \right\vert ^{m\left( 1-s^{\alpha
}\right) }dtds\right]
\end{eqnarray*}%
If $\ 0<k\leq 1,$ $0<m,n\leq 1$%
\begin{equation}
k^{m^{n}}\leq k^{mn}.  \label{1}
\end{equation}%
When $\eta =1,$ by (\ref{1}), we get%
\begin{eqnarray*}
&&\left\vert \frac{f\left( a\right) +f\left( b\right) }{2}-\frac{1}{b-a}%
\int_{a}^{b}f\left( x\right) dx\right\vert \\
&\leq &\frac{b-a}{2}\left\vert f^{\prime }\left( \frac{b}{m}\right)
\right\vert ^{m}\left[ \int_{0}^{1}\int_{0}^{1}\left\vert s-t\right\vert
dtds+\int_{0}^{1}\int_{0}^{1}\left\vert s-t\right\vert dtds\right] \\
&=&\frac{b-a}{3}\left\vert f^{\prime }\left( \frac{b}{m}\right) \right\vert
^{m}
\end{eqnarray*}%
When $\eta <1,$ by (\ref{1}), we get
\begin{eqnarray*}
&&\left\vert \frac{f\left( a\right) +f\left( b\right) }{2}-\frac{1}{b-a}%
\int_{a}^{b}f\left( x\right) dx\right\vert \\
&\leq &\frac{b-a}{2}\left\vert f^{\prime }\left( \frac{b}{m}\right)
\right\vert ^{m}\left[ \int_{0}^{1}\int_{0}^{1}\left\vert s-t\right\vert
\eta ^{\alpha t}dtds+\int_{0}^{1}\int_{0}^{1}\left\vert s-t\right\vert \eta
^{\alpha s}dtds\right] \\
&=&\frac{b-a}{2}\left\vert f^{\prime }\left( \frac{b}{m}\right) \right\vert
^{m}\left[ \frac{-\alpha ^{2}\ln ^{2}\eta -2\alpha \ln \eta +4\eta ^{\alpha
}+\alpha ^{2}\eta ^{\alpha }\ln ^{2}\eta -2\alpha \eta ^{\alpha }\ln \eta -4%
}{2\alpha ^{3}\ln ^{3}\eta }\right. \\
&&+\left. \frac{-\alpha \ln \eta +2\eta ^{\alpha }-\alpha \eta ^{\alpha }\ln
\eta -2}{2\alpha ^{2}\ln ^{2}\eta }\right]
\end{eqnarray*}%
which completes the proof.
\end{proof}

\begin{corollary}
Let$\ I\supset \left[ 0,\infty \right) $ be an open interval and let $f:\
I\rightarrow \left( 0,\infty \right) $ be a differentiable function on $I$
such that $f^{\prime }\in L\left( a,b\right) $ for $0\leq a<b<\infty .$ If $%
\left\vert f^{\prime }\left( x\right) \right\vert $ is $m$-logarithmically
convex on $\left[ 0,\frac{b}{m}\right] $ for $m\in \left( 0,1\right] $ , then%
\begin{equation*}
\left\vert \frac{f\left( a\right) +f\left( b\right) }{2}-\frac{1}{b-a}%
\int_{a}^{b}f\left( x\right) dx\right\vert \leq \left\{
\begin{array}{cc}
\frac{\left( b-a\right) }{3}\left\vert f^{\prime }\left( \frac{b}{m}\right)
\right\vert ^{m},\text{ \ \ \ \ \ \ \ \ \ \ \ \ \ \ \ \ \ \ \ \ \ } & \eta =1
\\
\frac{\left( b-a\right) }{2}\left\vert f^{\prime }\left( \frac{b}{m}\right)
\right\vert ^{m}\frac{-\ln ^{2}\eta -2\ln \eta +2\eta -2}{\ln ^{3}\eta }, &
\eta <1%
\end{array}%
\right.
\end{equation*}%
where $\eta $ is same as Theorem \ref{t1}.
\end{corollary}

\begin{corollary}
Let$\ I\supset \left[ 0,\infty \right) $ be an open interval and let $f:\
I\rightarrow \left( 0,\infty \right) $ be a differentiable function on $I$
such that $f^{\prime }\in L\left( a,b\right) $ for $0\leq a<b<\infty .$ If $%
\left\vert f^{\prime }\left( x\right) \right\vert $is $\alpha $%
-logarithmically convex on $\left[ 0,b\right] $ for $\alpha \in \left( 0,1%
\right] $ , then%
\begin{equation}
\left\vert \frac{f\left( a\right) +f\left( b\right) }{2}-\frac{1}{b-a}%
\int_{a}^{b}f\left( x\right) dx\right\vert \leq \left\{
\begin{array}{cc}
\frac{\left( b-a\right) }{3}\left\vert f^{\prime }\left( b\right)
\right\vert ,\text{ \ \ \ \ \ \ \ \ \ \ \ \ \ \ \ \ \ \ \ \ \ \ \ \ \ \ \ }
& \eta =1 \\
\frac{\left( b-a\right) }{2}\left\vert f^{\prime }\left( b\right)
\right\vert \frac{4\eta ^{\alpha }-4\alpha \ln \eta -2\alpha ^{2}\ln
^{2}\eta -4}{2\alpha ^{3}\ln ^{3}\eta }, & \eta <1%
\end{array}%
\right.  \notag
\end{equation}%
where $\eta =\left\vert f^{\prime }\left( a\right) \right\vert /\left\vert
f^{\prime }\left( b\right) \right\vert .$
\end{corollary}

\begin{theorem}
\label{t2}Let$\ I\supset \left[ 0,\infty \right) $ be an open interval and
let $f:\ I\rightarrow \left( 0,\infty \right) $ be a differentiable function
on $I$ such that $f^{\prime }\in L\left( a,b\right) $ for $0\leq a<b<\infty
. $ If $\left\vert f^{\prime }\left( x\right) \right\vert ^{q}$ is an $%
\left( \alpha ,m\right) $-logarithmically convex on $\left[ 0,\frac{b}{m}%
\right] $ for $\left( \alpha ,m\right) \in \left( 0,1\right] ^{2}$ and $%
p,q>1 $ with $\frac{1}{p}+\frac{1}{q}=1,$ then%
\begin{eqnarray}
&&\left\vert \frac{f\left( a\right) +f\left( b\right) }{2}-\frac{1}{b-a}%
\int_{a}^{b}f\left( x\right) dx\right\vert \\
&\leq &\left\{
\begin{array}{cc}
\left( b-a\right) \left\vert f^{\prime }\left( \frac{b}{m}\right)
\right\vert ^{m}\left( \frac{2}{\left( p+1\right) \left( p+2\right) }\right)
^{\frac{1}{p}},\text{ \ \ \ \ \ \ \ \ \ \ \ \ \ \ \ \ \ \ \ \ \ \ \ \ \ } &
\eta =1 \\
\left( b-a\right) \left\vert f^{\prime }\left( \frac{b}{m}\right)
\right\vert ^{m}\left( \frac{2}{\left( p+1\right) \left( p+2\right) }\right)
^{\frac{1}{p}}\times \left( \frac{\eta \left( \alpha q,\alpha q\right) -1}{%
\ln \eta \left( \alpha q,\alpha q\right) }\right) ^{\frac{1}{q}}, & \eta <1%
\end{array}%
\right.  \notag
\end{eqnarray}%
where $\eta \left( \alpha ,\alpha \right) $ is same as Theorem \ref{t1}
\end{theorem}

\begin{proof}
Since $\left\vert f^{\prime }\right\vert ^{q}$ is an $\left( \alpha
,m\right) $-logarithmically convex on $\left[ 0,\frac{b}{m}\right] $, from
Lemma \ref{l1} and the well known H\"{o}lder inequality, we have%
\begin{eqnarray}
&&  \label{y} \\
&&\left\vert \frac{f\left( a\right) +f\left( b\right) }{2}-\frac{1}{b-a}%
\int_{a}^{b}f\left( x\right) dx\right\vert  \notag \\
&\leq &\frac{b-a}{2}\int_{0}^{1}\int_{0}^{1}\left\vert \left( f^{\prime
}\left( ta+\left( 1-t\right) b\right) \right) -\left( f^{\prime }\left(
sa+\left( 1-s\right) b\right) \right) \right\vert \left\vert s-t\right\vert
dtds  \notag \\
&\leq &\frac{b-a}{2}\int_{0}^{1}\int_{0}^{1}\left\vert s-t\right\vert
\left\vert f^{\prime }\left( a\right) \right\vert ^{t^{\alpha }}\left\vert
f^{\prime }\left( \frac{b}{m}\right) \right\vert ^{m\left( 1-t^{\alpha
}\right) }dtds  \notag \\
&&+\frac{b-a}{2}\int_{0}^{1}\int_{0}^{1}\left\vert s-t\right\vert \left\vert
f^{\prime }\left( a\right) \right\vert ^{s^{\alpha }}\left\vert f^{\prime
}\left( \frac{b}{m}\right) \right\vert ^{m\left( 1-s^{\alpha }\right) }dtds
\notag \\
&\leq &\frac{b-a}{2}\left\vert f^{\prime }\left( \frac{b}{m}\right)
\right\vert ^{m}\left( \int_{0}^{1}\int_{0}^{1}\left\vert s-t\right\vert
^{p}dtds\right) ^{\frac{1}{p}}  \notag \\
&&\times \left[ \left( \int_{0}^{1}\int_{0}^{1}\eta ^{qt^{\alpha
}}dtds\right) ^{\frac{1}{q}}+\left( \int_{0}^{1}\int_{0}^{1}\eta
^{qs^{\alpha }}dtds\right) ^{\frac{1}{q}}\right]  \notag
\end{eqnarray}%
If $\ \eta =1,$ by (\ref{1}), we obtain%
\begin{eqnarray*}
&&\left\vert \frac{f\left( a\right) +f\left( b\right) }{2}-\frac{1}{b-a}%
\int_{a}^{b}f\left( x\right) dx\right\vert \\
&\leq &\left( b-a\right) \left\vert f^{\prime }\left( \frac{b}{m}\right)
\right\vert ^{m}\left( \int_{0}^{1}\int_{0}^{1}\left\vert s-t\right\vert
^{p}dtds\right) ^{\frac{1}{p}} \\
&=&\left( b-a\right) \left\vert f^{\prime }\left( \frac{b}{m}\right)
\right\vert ^{m}\left( \frac{2}{\left( p+1\right) \left( p+2\right) }\right)
^{\frac{1}{p}}
\end{eqnarray*}%
If $\eta <1,$ by (\ref{1}), we obtain%
\begin{eqnarray}
&&\left\vert \frac{f\left( a\right) +f\left( b\right) }{2}-\frac{1}{b-a}%
\int_{a}^{b}f\left( x\right) dx\right\vert \\
&\leq &\frac{b-a}{2}\left\vert f^{\prime }\left( \frac{b}{m}\right)
\right\vert ^{m}\left( \int_{0}^{1}\int_{0}^{1}\left\vert s-t\right\vert
^{p}dtds\right) ^{\frac{1}{p}}  \notag \\
&&\times \left[ \left( \int_{0}^{1}\int_{0}^{1}\eta ^{qt^{\alpha
}}dtds\right) ^{\frac{1}{q}}+\left( \int_{0}^{1}\int_{0}^{1}\eta
^{qs^{\alpha }}dtds\right) ^{\frac{1}{q}}\right]  \notag \\
&=&\left( b-a\right) \left\vert f^{\prime }\left( \frac{b}{m}\right)
\right\vert ^{m}\left( \frac{2}{\left( p+1\right) \left( p+2\right) }\right)
^{\frac{1}{p}}\times \left( \frac{\eta \left( \alpha q,\alpha q\right) -1}{%
\ln \eta \left( \alpha q,\alpha q\right) }\right) ^{\frac{1}{q}}  \notag
\end{eqnarray}

\textit{which completes the proof.}
\end{proof}

\begin{corollary}
Let$\ I\supset \left[ 0,\infty \right) $ be an open interval and let
$f:\ I\rightarrow \left( 0,\infty \right) $ be a differentiable
function on $I$ such that $f^{\prime }\in L\left( a,b\right) $ for
$0\leq a<b<\infty .$ If $\left\vert f^{\prime }\left( x\right)
\right\vert ^{q}$is an $m$-logarithmically convex on $\left[ 0,\frac{b}{m}%
\right] $ for $m\in \left( 0,1\right] $ and $p=q=2,$ then%
\begin{equation}
\left\vert \frac{f\left( a\right) +f\left( b\right) }{2}-\frac{1}{b-a}%
\int_{a}^{b}f\left( x\right) dx\right\vert \leq \left\{
\begin{array}{cc}
\left( b-a\right) \left\vert f^{\prime }\left( \frac{b}{m}\right)
\right\vert ^{m}\sqrt{\frac{1}{6}},\text{ \ \ \ \ \ \ \ \ \ \ \ \ \ \ \ \ \
\ } & \eta =1 \\
\left( b-a\right) \left\vert f^{\prime }\left( \frac{b}{m}\right)
\right\vert ^{m}\sqrt{\frac{1}{6}}\left( \frac{\eta \left( 2,2\right) -1}{%
\ln \eta \left( 2,2\right) }\right) ^{\frac{1}{2}}, & \eta <1%
\end{array}%
\right.  \notag
\end{equation}
\end{corollary}

\begin{corollary}
Let$\ I\supset \left[ 0,\infty \right) $ be an open interval and let $f:\
I\rightarrow \left( 0,\infty \right) $ be a differentiable function on $I$
such that $f^{\prime }\in L\left( a,b\right) $ for $0\leq a<b<\infty .$ If $%
\left\vert f^{\prime }\left( x\right) \right\vert $is $\alpha $%
-logarithmically convex on $\left[ 0,b\right] $ for $\alpha \in \left( 0,1%
\right] $ , then%
\begin{eqnarray*}
&&\left\vert \frac{f\left( a\right) +f\left( b\right) }{2}-\frac{1}{b-a}%
\int_{a}^{b}f\left( x\right) dx\right\vert \\
&\leq &\left\{
\begin{array}{cc}
\left( b-a\right) \left\vert f^{\prime }\left( b\right) \right\vert \left(
\frac{2}{\left( p+1\right) \left( p+2\right) }\right) ^{\frac{1}{p}},\text{
\ \ \ \ \ \ \ \ \ \ \ \ \ \ \ \ \ \ \ \ \ } & \eta =1 \\
\left( b-a\right) \left\vert f^{\prime }\left( b\right) \right\vert \left(
\frac{2}{\left( p+1\right) \left( p+2\right) }\right) ^{\frac{1}{p}}\left(
\frac{\eta \left( \alpha q,\alpha q\right) -1}{\ln \eta \left( \alpha
q,\alpha q\right) }\right) ^{\frac{1}{q}}, & \eta <1%
\end{array}%
\right.
\end{eqnarray*}

where $\eta =\left\vert f^{\prime }\left( a\right) \right\vert /\left\vert
f^{\prime }\left( b\right) \right\vert .$
\end{corollary}

\begin{theorem}
Let$\ I\supset \left[ 0,\infty \right) $ be an open interval and let $f:\
I\rightarrow \left( 0,\infty \right) $ be a differentiable function on $I$
such that $f^{\prime }\in L\left( a,b\right) $ for $0\leq a<b<\infty .$ If $%
\left\vert f^{\prime }\left( x\right) \right\vert ^{q}$ is $\left( \alpha
,m\right) $-logarithmically convex on $\left[ 0,\frac{b}{m}\right] $ for $%
\left( \alpha ,m\right) \in \left( 0,1\right] ^{2}$, and then%
\begin{eqnarray}
&& \\
&&\left\vert \frac{f\left( a\right) +f\left( b\right) }{2}-\frac{1}{b-a}%
\int_{a}^{b}f\left( x\right) dx\right\vert  \notag \\
&\leq &\left\{
\begin{array}{c}
\frac{b-a}{3}\left\vert f^{\prime }\left( \frac{b}{m}\right) \right\vert
^{m},\text{ \ \ \ \ \ \ \ \ \ \ \ \ \ \ \ \ \ \ \ \ \ \ \ \ \ \ \ \ \ \ \ \
\ \ \ \ \ \ \ \ \ \ \ \ \ \ \ \ \ \ \ \ \ \ \ \ \ \ \ \ \ \ \ \ \ \ \ \ \ }%
\eta =1 \\
\frac{b-a}{2}\left( \frac{1}{3}\right) ^{1-\frac{1}{q}}\left\vert f^{\prime
}\left( \frac{b}{m}\right) \right\vert ^{m}\left\{ \left[ \frac{2\varphi -2}{%
\left[ \ln \varphi \right] ^{3}}-\frac{\varphi +1}{\left[ \ln \varphi \right]
^{2}}-\frac{1-\varphi }{2\ln \varphi }\right] ^{\frac{1}{q}}+\left[ \frac{%
\varphi -1}{\left[ \ln \varphi \right] ^{2}}-\frac{\varphi +1}{2\ln \varphi }%
\right] ^{\frac{1}{q}}\right\} ,\text{\ }\eta <1%
\end{array}%
\right.  \notag
\end{eqnarray}%
where $\eta \left( \alpha ,\alpha \right) $ is same as Theorem \ref{t1}, and
we take $\eta \left( \alpha q,\alpha q\right) =\varphi $ .$\allowbreak $
\end{theorem}

\begin{proof}
Since $\left\vert f^{\prime }\right\vert ^{q}$ is an $\left( \alpha
,m\right) $-logarithmically convex on $\left[ 0,\frac{b}{m}\right] $, for $%
q\geq 1$, from Lemma \ref{l1} and the well known power mean integral
inequality, we get
\begin{eqnarray*}
&&\left\vert \frac{f\left( a\right) +f\left( b\right) }{2}-\frac{1}{b-a}%
\int_{a}^{b}f\left( x\right) dx\right\vert \\
&\leq &\frac{b-a}{2}\int_{0}^{1}\int_{0}^{1}\left\vert \left( f^{\prime
}\left( ta+\left( 1-t\right) b\right) \right) -\left( f^{\prime }\left(
sa+\left( 1-s\right) b\right) \right) \right\vert \left\vert s-t\right\vert
dtds \\
&\leq &\frac{b-a}{2}\left( \int_{0}^{1}\int_{0}^{1}\left\vert s-t\right\vert
dtds\right) ^{1-\frac{1}{q}}\left( \int_{0}^{1}\int_{0}^{1}\left\vert
s-t\right\vert \left\vert f^{\prime }\left( ta+\left( 1-t\right) b\right)
\right\vert ^{q}dtds\right) ^{\frac{1}{q}} \\
&&+\frac{b-a}{2}\left( \int_{0}^{1}\int_{0}^{1}\left\vert s-t\right\vert
dtds\right) ^{1-\frac{1}{q}}\left( \int_{0}^{1}\int_{0}^{1}\left\vert
s-t\right\vert \left\vert f^{\prime }\left( sa+\left( 1-s\right) b\right)
\right\vert ^{q}dtds\right) ^{\frac{1}{q}} \\
&\leq &\frac{b-a}{2}\left\vert f^{\prime }\left( \frac{b}{m}\right)
\right\vert ^{m}\left( \int_{0}^{1}\int_{0}^{1}\left\vert s-t\right\vert
dtds\right) ^{1-\frac{1}{q}}\left( \int_{0}^{1}\int_{0}^{1}\left\vert
s-t\right\vert \eta ^{qt^{\alpha }}dtds\right) ^{\frac{1}{q}} \\
&&+\frac{b-a}{2}\left\vert f^{\prime }\left( \frac{b}{m}\right) \right\vert
^{m}\left( \int_{0}^{1}\int_{0}^{1}\left\vert s-t\right\vert dtds\right) ^{1-%
\frac{1}{q}}\left( \int_{0}^{1}\int_{0}^{1}\left\vert s-t\right\vert \eta
^{qs^{\alpha }}dtds\right) ^{\frac{1}{q}}
\end{eqnarray*}%
When $\eta =1,$ by (\ref{1}), we obtain
\begin{eqnarray*}
&&\left\vert \frac{f\left( a\right) +f\left( b\right) }{2}-\frac{1}{b-a}%
\int_{a}^{b}f\left( x\right) dx\right\vert \\
&\leq &\frac{b-a}{2}\left( \frac{1}{3}\right) ^{1-\frac{1}{q}}\left\vert
f^{\prime }\left( \frac{b}{m}\right) \right\vert ^{m}\left(
\int_{0}^{1}\int_{0}^{1}\left\vert s-t\right\vert dtds\right) ^{\frac{1}{q}}
\\
&&+\frac{b-a}{2}\left( \frac{1}{3}\right) ^{1-\frac{1}{q}}\left\vert
f^{\prime }\left( \frac{b}{m}\right) \right\vert ^{m}\left(
\int_{0}^{1}\int_{0}^{1}\left\vert s-t\right\vert dtds\right) ^{\frac{1}{q}}
\\
&=&\frac{b-a}{3}\left\vert f^{\prime }\left( \frac{b}{m}\right) \right\vert
^{m}
\end{eqnarray*}%
When $\eta <1,$ by (\ref{1}), we obtain%
\begin{eqnarray*}
&&\left\vert \frac{f\left( a\right) +f\left( b\right) }{2}-\frac{1}{b-a}%
\int_{a}^{b}f\left( x\right) dx\right\vert \\
&\leq &\frac{b-a}{2}\left( \frac{1}{3}\right) ^{1-\frac{1}{q}}\left\vert
f^{\prime }\left( \frac{b}{m}\right) \right\vert ^{m}\left(
\int_{0}^{1}\int_{0}^{1}\left\vert s-t\right\vert \eta ^{\alpha
qt}dtds\right) ^{\frac{1}{q}} \\
&&+\frac{b-a}{2}\left( \frac{1}{3}\right) ^{1-\frac{1}{q}}\left\vert
f^{\prime }\left( \frac{b}{m}\right) \right\vert ^{m}\left(
\int_{0}^{1}\int_{0}^{1}\left\vert s-t\right\vert \eta ^{\alpha
qs}dtds\right) ^{\frac{1}{q}} \\
&=&\frac{b-a}{2}\left( \frac{1}{3}\right) ^{1-\frac{1}{q}}\left\vert
f^{\prime }\left( \frac{b}{m}\right) \right\vert ^{m} \\
&&\times \left\{ \left[ \frac{2\eta \left( \alpha q,\alpha q\right) -2}{%
\left[ \ln \left( \eta \left( \alpha q,\alpha q\right) \right) \right] ^{3}}-%
\frac{\eta \left( \alpha q,\alpha q\right) +1}{\left[ \ln \left( \eta \left(
\alpha q,\alpha q\right) \right) \right] ^{2}}-\frac{1-\eta \left( \alpha
q,\alpha q\right) }{2\ln \left( \eta \left( \alpha q,\alpha q\right) \right)
}\right] ^{\frac{1}{q}}\right. \\
&&+\left. \left[ \frac{\eta \left( \alpha q,\alpha q\right) -1}{\left[ \ln
\left( \eta \left( \alpha q,\alpha q\right) \right) \right] ^{2}}-\frac{\eta
\left( \alpha q,\alpha q\right) +1}{2\ln \left( \eta \left( \alpha q,\alpha
q\right) \right) }\right] ^{\frac{1}{q}}\right\}
\end{eqnarray*}%
\textit{which completes the proof.}
\end{proof}

\begin{corollary}
Let$\ I\supset \left[ 0,\infty \right) $ be an open interval and let $f:\
I\rightarrow \left( 0,\infty \right) $ be a differentiable function on $I$
such that $f^{\prime }\in L\left( a,b\right) $ for $0\leq a<b<\infty .$ If $%
\left\vert f^{\prime }\left( x\right) \right\vert ^{q}$is $m$%
-logarithmically convex on $\left[ 0,\frac{b}{m}\right] $ for $m\in \left(
0,1\right] $, then%
\begin{eqnarray*}
&&\left\vert \frac{f\left( a\right) +f\left( b\right) }{2}-\frac{1}{b-a}%
\int_{a}^{b}f\left( x\right) dx\right\vert \\
&\leq &\left\{
\begin{array}{c}
\frac{b-a}{3}\left\vert f^{\prime }\left( \frac{b}{m}\right) \right\vert
^{m},\text{ \ \ \ \ \ \ \ \ \ \ \ \ \ \ \ \ \ \ \ \ \ \ \ \ \ \ \ \ \ \ \ \
\ \ \ \ \ \ \ \ \ \ \ \ \ \ \ \ \ \ \ \ \ \ \ \ \ \ \ }\eta =1 \\
\frac{\left( b-a\right) }{2}\left( \frac{1}{3}\right) ^{1-\frac{1}{q}%
}\left\vert f^{\prime }\left( \frac{b}{m}\right) \right\vert ^{m}\text{ \ \
\ \ \ \ \ \ \ \ \ \ \ \ \ \ \ \ \ \ \ \ \ \ \ \ \ \ \ \ \ \ \ \ \ \ \ \ \ \
\ \ \ \ \ \ \ \ \ \ \ \ \ \ \ \ \ \ \ \ \ \ \ \ \ \ \ \ \ \ \ \ \ \ \ \ \ \ }
\\
\times \left\{ \left[ \frac{2\eta \left( q,q\right) -2}{\left[ \ln \eta
\left( q,q\right) \right] ^{3}}-\frac{\eta \left( q,q\right) +1}{\left[ \ln
\eta \left( q,q\right) \right] ^{2}}-\frac{1-\eta \left( q,q\right) }{2\ln
\eta \left( q,q\right) }\right] ^{\frac{1}{q}}+\left[ \frac{\eta \left(
q,q\right) -1}{\left[ \ln \eta \left( q,q\right) \right] ^{2}}-\frac{\eta
\left( q,q\right) +1}{2\ln \eta \left( q,q\right) }\right] ^{\frac{1}{q}%
}\right\} ,\eta <1%
\end{array}%
\right.
\end{eqnarray*}
\end{corollary}

\begin{corollary}
Let$\ I\supset \left[ 0,\infty \right) $ be an open interval and let $f:\
I\rightarrow \left( 0,\infty \right) $ be a differentiable function on $I$
such that $f^{\prime }\in L\left( a,b\right) $ for $0\leq a<b<\infty .$ If $%
\left\vert f^{\prime }\left( x\right) \right\vert $ is $\alpha $%
-logarithmically convex on $\left[ 0,b\right] $ for $\alpha \in \left( 0,1%
\right] $ , then

\begin{eqnarray*}
&&\left\vert \frac{f\left( a\right) +f\left( b\right) }{2}-\frac{1}{b-a}%
\int_{a}^{b}f\left( x\right) dx\right\vert \\
&\leq &\left\{
\begin{array}{cc}
\left( b-a\right) \left( \frac{1}{3}\right) \left\vert f^{\prime }\left(
b\right) \right\vert ,\text{ \ \ \ \ \ \ \ \ \ \ \ \ \ \ \ \ \ \ \ \ \ \ \ \
\ \ \ \ \ \ \ \ \ \ \ \ \ \ \ \ \ \ \ \ \ \ \ \ \ \ \ \ \ } & \eta =1 \\
\begin{array}{c}
\frac{\left( b-a\right) }{2}\left( \frac{1}{3}\right) ^{1-\frac{1}{q}%
}\left\vert f^{\prime }\left( b\right) \right\vert \left( \left[ \frac{2\eta
\left( \alpha q,\alpha q\right) -2}{\left[ \ln \left( \eta \left( \alpha
q,\alpha q\right) \right) \right] ^{3}}-\frac{\eta \left( \alpha q,\alpha
q\right) +1}{\left[ \ln \left( \eta \left( \alpha q,\alpha q\right) \right) %
\right] ^{2}}-\frac{1-\eta \left( \alpha q,\alpha q\right) }{2\ln \left(
\eta \left( \alpha q,\alpha q\right) \right) }\right] ^{\frac{1}{q}}\right.
\\
+\left. \left[ \frac{\eta \left( \alpha q,\alpha q\right) -1}{\left[ \ln
\left( \eta \left( \alpha q,\alpha q\right) \right) \right] ^{2}}-\frac{\eta
\left( \alpha q,\alpha q\right) +1}{2\ln \left( \eta \left( \alpha q,\alpha
q\right) \right) }\right] ^{\frac{1}{q}}\right)%
\end{array}%
, & \eta <1%
\end{array}%
\right.
\end{eqnarray*}

where $\eta =\left\vert f^{\prime }\left( a\right) \right\vert /\left\vert
f^{\prime }\left( b\right) \right\vert .$
\end{corollary}

\begin{theorem}
\label{t4}Let $f:I\subset
\mathbb{R}
_{+}\rightarrow
\mathbb{R}
_{+}$ be differentiable on $I^{\circ },$ $a,b\in I$, with $a<b$ and $%
f^{\prime }\in L\left( \left[ a,b\right] \right) .$ If $\left\vert f^{\prime
}\right\vert $ is an $\left( \alpha ,m\right) $-logarithmically convex $%
\left[ 0,\frac{b}{m}\right] $ for $\left( \alpha ,m\right) \in \left( 0,1%
\right] ^{2}$ and $\mu _{1},\mu _{2},\tau _{1},\tau _{2}>0$ with $\mu
_{1}+\tau _{1}=1$ and $\mu _{2}+\tau _{2}=1$, then%
\begin{eqnarray}
&& \\
&&\left\vert \frac{f\left( a\right) +f\left( b\right) }{2}-\frac{1}{b-a}%
\int_{a}^{b}f\left( x\right) dx\right\vert  \notag \\
&\leq &\left\{
\begin{array}{cc}
\frac{\left( b-a\right) }{2}\left\vert f^{\prime }\left( \frac{b}{m}\right)
\right\vert ^{m}\left[ \frac{2\mu _{1}^{3}}{\left( 2\mu _{1}+1\right) \left(
\mu _{1}+1\right) }+\frac{2\mu _{2}^{3}}{\left( 2\mu _{2}+1\right) \left(
\mu _{2}+1\right) }+\tau _{1}+\tau _{2}\right] , & \eta =1 \\
\begin{array}{c}
\frac{\left( b-a\right) }{2}\left\vert f^{\prime }\left( \frac{b}{m}\right)
\right\vert ^{m}\left\{ \frac{2\mu _{1}^{3}}{\left( 2\mu _{1}+1\right)
\left( \mu _{1}+1\right) }+\frac{2\mu _{2}^{3}}{\left( 2\mu _{2}+1\right)
\left( \mu _{2}+1\right) }\right. \\
+\left. \tau _{1}\frac{\eta \left( \frac{\alpha }{\tau _{1}},\frac{\alpha }{%
\tau _{1}}\right) -1}{\ln \eta \left( \frac{\alpha }{\tau _{1}},\frac{\alpha
}{\tau _{1}}\right) }+\tau _{2}\frac{\eta \left( \frac{\alpha }{\tau _{2}},%
\frac{\alpha }{\tau _{2}}\right) -1}{\ln \eta \left( \frac{\alpha }{\tau _{2}%
},\frac{\alpha }{\tau _{2}}\right) }\right\} ,%
\end{array}%
\text{ \ \ \ \ \ \ \ \ \ \ \ \ } & \eta <1%
\end{array}%
\right.  \notag
\end{eqnarray}%
where $\eta \left( \alpha ,\alpha \right) $ is same as Theorem \ref{t1}.
\end{theorem}

\begin{proof}
Since $\left\vert f^{\prime }\right\vert ^{q}$ is an $\left( \alpha
,m\right) $-logarithmically convex on $\left[ 0,\frac{b}{m}\right] $, from
Lemma \ref{l1}, we have%
\begin{eqnarray}
&&\left\vert \frac{f\left( a\right) +f\left( b\right) }{2}-\frac{1}{b-a}%
\int_{a}^{b}f\left( x\right) dx\right\vert  \label{f} \\
&\leq &\frac{\left( b-a\right) }{2}\int_{0}^{1}\int_{0}^{1}\left\vert \left(
f^{\prime }\left( ta+\left( 1-t\right) b\right) \right) -\left( f^{\prime
}\left( sa+\left( 1-s\right) b\right) \right) \right\vert \left\vert
s-t\right\vert dtds  \notag \\
&\leq &\frac{\left( b-a\right) }{2}\left[ \int_{0}^{1}\int_{0}^{1}\left\vert
s-t\right\vert \left\vert f^{\prime }\left( a\right) \right\vert ^{t^{\alpha
}}\left\vert f^{\prime }\left( \frac{b}{m}\right) \right\vert ^{m\left(
1-t^{\alpha }\right) }dtds\right]  \notag \\
&&+\frac{\left( b-a\right) }{2}\left[ \int_{0}^{1}\int_{0}^{1}\left\vert
s-t\right\vert \left\vert f^{\prime }\left( a\right) \right\vert ^{s^{\alpha
}}\left\vert f^{\prime }\left( \frac{b}{m}\right) \right\vert ^{m\left(
1-s^{\alpha }\right) }dtds\right]  \notag \\
&=&\frac{\left( b-a\right) }{2}\left\vert f^{\prime }\left( \frac{b}{m}%
\right) \right\vert ^{m}\left[ \int_{0}^{1}\int_{0}^{1}\left\vert
s-t\right\vert \eta ^{t^{\alpha }}dtds+\int_{0}^{1}\int_{0}^{1}\left\vert
s-t\right\vert \eta ^{s^{\alpha }}dtds\right]  \notag
\end{eqnarray}%
for all $t\in \left[ 0,1\right] .$ Using the well known inequality $rt\leq
\mu r^{\frac{1}{\mu }}+\tau t^{\frac{1}{\tau }},$ on the right side of (\ref%
{f}), we get%
\begin{eqnarray}
&&\int_{0}^{1}\int_{0}^{1}\left\vert s-t\right\vert \eta ^{t^{\alpha
}}dtds+\int_{0}^{1}\int_{0}^{1}\left\vert s-t\right\vert \eta ^{s^{\alpha
}}dtds  \label{w} \\
&\leq &\mu _{1}\int_{0}^{1}\int_{0}^{1}\left\vert s-t\right\vert ^{\frac{1}{%
\mu _{1}}}dtds+\tau _{1}\int_{0}^{1}\int_{0}^{1}\eta ^{\frac{t^{\alpha }}{%
\tau _{1}}}dtds  \notag \\
&&+\mu _{2}\int_{0}^{1}\int_{0}^{1}\left\vert s-t\right\vert ^{\frac{1}{\mu
_{2}}}dtds+\tau _{2}\int_{0}^{1}\int_{0}^{1}\eta ^{\frac{s^{\alpha }}{\tau
_{2}}}dtds  \notag
\end{eqnarray}%
When $\eta =1,$ by (\ref{1}), we get%
\begin{eqnarray}
&&\int_{0}^{1}\int_{0}^{1}\left\vert s-t\right\vert \eta ^{t^{\alpha
}}dtds+\int_{0}^{1}\int_{0}^{1}\left\vert s-t\right\vert \eta ^{s^{\alpha
}}dtds  \label{e} \\
&\leq &\frac{2\mu _{1}^{3}}{\left( 2\mu _{1}+1\right) \left( \mu
_{1}+1\right) }+\frac{2\mu _{2}^{3}}{\left( 2\mu _{2}+1\right) \left( \mu
_{2}+1\right) }+\tau _{1}+\tau _{2}  \notag
\end{eqnarray}%
When $\eta <1,$ by (\ref{1}), we get%
\begin{eqnarray}
&&  \label{r} \\
&&\int_{0}^{1}\int_{0}^{1}\left\vert s-t\right\vert \eta ^{t^{\alpha
}}dtds+\int_{0}^{1}\int_{0}^{1}\left\vert s-t\right\vert \eta ^{s^{\alpha
}}dtds  \notag \\
&\leq &\mu _{1}\int_{0}^{1}\int_{0}^{1}\left\vert s-t\right\vert ^{\frac{1}{%
\mu _{1}}}dtds+\tau _{1}\int_{0}^{1}\int_{0}^{1}\eta ^{\frac{t^{\alpha }}{%
\tau _{1}}}dtds+\mu _{2}\int_{0}^{1}\int_{0}^{1}\left\vert s-t\right\vert ^{%
\frac{1}{\mu _{2}}}dtds+\tau _{2}\int_{0}^{1}\int_{0}^{1}\eta ^{\frac{%
s^{\alpha }}{\tau _{2}}}dtds  \notag \\
&\leq &\mu _{1}\int_{0}^{1}\int_{0}^{1}\left\vert s-t\right\vert ^{\frac{1}{%
\mu _{1}}}dtds+\mu _{2}\int_{0}^{1}\int_{0}^{1}\left\vert s-t\right\vert ^{%
\frac{1}{\mu _{2}}}dtds+\tau _{1}\int_{0}^{1}\int_{0}^{1}\eta ^{\frac{\alpha
t}{\tau _{1}}}dtds+\tau _{2}\int_{0}^{1}\int_{0}^{1}\eta ^{\frac{\alpha s}{%
\tau _{2}}}dtds  \notag \\
&=&\frac{2\mu _{1}^{3}}{\left( 2\mu _{1}+1\right) \left( \mu _{1}+1\right) }+%
\frac{2\mu _{2}^{3}}{\left( 2\mu _{2}+1\right) \left( \mu _{2}+1\right) }%
+\tau _{1}\frac{\eta \left( \frac{\alpha }{\tau _{1}},\frac{\alpha }{\tau
_{1}}\right) -1}{\ln \eta \left( \frac{\alpha }{\tau _{1}},\frac{\alpha }{%
\tau _{1}}\right) }+\tau _{2}\frac{\eta \left( \frac{\alpha }{\tau _{2}},%
\frac{\alpha }{\tau _{2}}\right) -1}{\ln \eta \left( \frac{\alpha }{\tau _{2}%
},\frac{\alpha }{\tau _{2}}\right) }  \notag
\end{eqnarray}

from (\ref{f})-(\ref{r}), which completes the proof.
\end{proof}

\begin{corollary}
Under the assumptions of Theorem \ref{t4}, and $\mu =\mu _{1}=\mu _{2}>0,$ $%
\tau =\tau _{1}=\tau _{2}>0$ with $\mu +\tau =1,$ then we have%
\begin{equation*}
\left\vert \frac{f\left( a\right) +f\left( b\right) }{2}-\frac{1}{b-a}%
\int_{a}^{b}f\left( x\right) dx\right\vert \leq \left\{
\begin{array}{cc}
\frac{\left( b-a\right) }{2}\left\vert f^{\prime }\left( \frac{b}{m}\right)
\right\vert ^{m}\left[ \frac{4\mu ^{3}}{\left( 2\mu +1\right) \left( \mu
+1\right) }+2\tau \right] , & \eta =1 \\
\frac{\left( b-a\right) }{2}\left\vert f^{\prime }\left( \frac{b}{m}\right)
\right\vert ^{m}\left[ \frac{4\mu ^{3}}{\left( 2\mu +1\right) \left( \mu
+1\right) }+2\tau \frac{\eta \left( \frac{\alpha }{\tau },\frac{\alpha }{%
\tau }\right) -1}{\ln \eta \left( \frac{\alpha }{\tau },\frac{\alpha }{\tau }%
\right) }\right] , & \eta <1%
\end{array}%
\right.
\end{equation*}
\end{corollary}

\end{document}